\documentclass[10pt]{amsart}\oddsidemargin -1mm \evensidemargin -1mm \topmargin -15mm\headheight5mm \headsep 3.5mm \textheight 245mm\textwidth 170mm 
\usepackage{amsfonts, amssymb, amsthm, amsmath, euscript}
\newtheorem{theorem}{Theorem}[section]\newtheorem{lemma}[theorem]{Lemma}\newtheorem{remark}[theorem]{Remark}\newtheorem{proposition}[theorem]{Proposition}\newtheorem{corollary}[theorem]{Corollary}\newtheorem*{Ex0}{Examples}
\DeclareMathOperator{\Ue}{\mathrm{U}}\DeclareMathOperator{\Sa}{\mathrm{S}}\DeclareMathOperator{\Var}{\mathrm{Var}}\DeclareMathOperator{\Ann}{\mathrm{Ann}}\DeclareMathOperator{\Irr}{\mathrm{Irr}}\DeclareMathOperator{\Hom}{\mathrm{Hom}}
\begin{document}
\title{Primitive ideals of $\Ue(\frak{sl}(\infty))$}
\author{Ivan Penkov, Alexey Petukhov}
\address{Ivan Penkov: Jacobs University Bremen, Campus Ring 1, D-28759, Bremen, Germany}
\email{i.penkov@jacobs-university.de}
\address{Alexey Petukhov: The University of Manchester, Oxford Road M13 9PL, Manchester, UK, on leave from the Institute for Information Transmission Problems, Bolshoy Karetniy 19-1, Moscow 127994, Russia}
\email{alex-{}-2@yandex.ru}
\maketitle

\begin{abstract}We provide an explicit description of the primitive ideals of the enveloping algebra $\Ue(\frak{sl}(\infty))$ of the infinite-dimensional finitary Lie algebra $\frak{sl}(\infty)$ over an uncountable algebraically closed field of characteristic 0. Our main new result is that any primitive ideal of $\Ue(\frak{sl}(\infty))$ is integrable. A classification of integrable primitive ideals of $\Ue(\frak{sl}(\infty))$ has been known previously, and relies on the pioneering work of A.~Zhilinskii.


\end{abstract}

\section{Introduction} A two-sided ideal $I$ of an associative algebra $A$ is called {\it primitive} if $I$ is the annihilator of a simple $A$-module. Given an infinite-dimensional associative algebra $A$,
it may be too hard to classify simple $A$-modules (this problem seems to be open for the algebra of differential operators in two variables) but it may still  be possible to provide an explicit  description of the primitive ideals of $A$. This is precisely the situation when $A=\Ue(\frak g)$ is the enveloping algebra of a semisimple finite-dimensional Lie algebra over a field of characteristic 0. In this case, a rough description of primitive ideals is given by the celebrated Duflo Theorem~\cite{Du}: it claims that every primitive ideal of $A=\Ue(\frak g)$ is the annihilator of a simple highest weight module. This reduces the problem of classifying primitive ideals to a combinatorial problem. This latter problem has been solved due to the efforts of many mathematicians, in particular D.~Barbasch, D.~Vogan~\cite{BV1, BV2}, W.~Borho, J.-C.~Jantzen~\cite{BJ}, A.~Joseph~\cite{Jo}, D.~Kazhdan,~G.~Lusztig~\cite{Lu}, and others, see, for example,~\cite{Atl}. In~\cite{PP3} we have made an attempt to summarize the  combinatorial description of the primitive  ideals of $\Ue(\frak g)$ for classical simple Lie algebras $\frak g$ in a language suitable for studying the case when $\operatorname{rk}\frak g$ tends to $\infty$.

In the present paper we consider the Lie algebra $\frak{sl}(\infty)$ which consists of traceless finitary infinite matrices, i.e. traceless infinite  matrices each of which has only finitely many nonzero entries. Our main result is an explicit description of all primitive ideals of the enveloping algebra $\Ue(\frak{sl}(\infty))$. Very roughly, when passing from $\Ue(\frak{sl}(n))$ to $\Ue(\frak{sl}(\infty))$, the problem of classifying irreducible representations cannot become easier, while, as we show in this paper, the problem of classifying primitive ideals admits a beautifully simple answer. In contrast with the case of the finite-dimensional simple Lie  algebra, $\Ue(\frak{sl}(\infty))$ has only countably many primitive ideals. This is related to the circumstance that a generic irreducible highest weight $\Ue(\frak{sl}(\infty))$-module has zero annihilator, see~\cite{PP3}. We intend to use our detailed understanding of primitive ideals of $\Ue(\frak{sl}(\infty))$ in the ongoing development of the representation theory of the Lie algebra $\frak{sl}(\infty)$.

Moreover, in the recent review~\cite{PP3} we presented the classification of primitive ideals of $\Ue(\frak{sl}(\infty))$ subject to the condition that they are integrable. This classification relies on the work of A.~Zhilinskii~\cite{Zh1, Zh2, Zh3}. We recall that a two-sided ideal $I$ is {\it integrable} if  $I=\Ann_{\Ue(\frak{sl}(\infty))}M$ for an integrable $\Ue(\frak{sl}(\infty))$-module $M$, i.e., a $\Ue(\frak{sl}(\infty))$-module $M$ which becomes a sum of finite-dimensional $\Ue(\frak{sl}(n))$-modules after being restricted to $\Ue(\frak{sl}(n))$ for each $n\ge2$. The main result of the present paper is that every primitive ideal of $\Ue(\frak{sl}(\infty))$ is integrable, and hence the integrable primitive ideals described in~\cite{PP3} are all primitive ideals.

\section{Main result}\label{Smr} Fix an uncountable algebraically closed field $\mathbb F$ of characteristic 0. All vector spaces (in particular, Lie algebras) are assumed to be defined over $\mathbb F$. If $W$ is a vector space, then $W^*:=\operatorname{Hom}_{\mathbb F}(W, \mathbb F)$.

One can define the Lie algebra $\frak{sl}(\infty)$ as the direct limit of (arbitrary) inclusions of the form
$$\frak{sl}(2)\hookrightarrow\frak{sl}(3)\hookrightarrow...\hookrightarrow\frak{sl}(n)\hookrightarrow....$$
It is well known that this property determines $\frak{sl}(\infty)$ up to isomorphism. Moreover, for $n\ge3$, the defining representation $V(n)$ of $\frak{sl}(n)$ decomposes as $V(n-1)\oplus\mathbb F$ over $\frak{sl}(n-1)$, and, up to isomorphism, there is only one $\frak{sl}(\infty)$-module obtained as a direct limit $\varinjlim V(n)$ (this is not the case for other infinite-dimensional locally simple Lie algebras as can be seen for instance from~\cite{HS}). The direct imit $\varinjlim V(n)$ is by definition the {\it natural} $\frak{sl}(\infty)$-module, and we denote it by $V$. Similarly, there is a well-defined {\it  conatural} $\frak{sl}(\infty)$-module $V_*:=\varinjlim (V(n)^*)$. In what follows we consider also the symmetric and exterior algebras $\Sa^\cdot(V):=\oplus_{k\ge0}\Sa^k(V)$ and $\Lambda^\cdot(V):=\oplus_{k\ge0}\Lambda^k(V)$, as well as $\Sa^\cdot(V_*)$ and $\Lambda^\cdot(V_*)$.

Next, for any (possibly empty) Young diagram $Y$ whose column lengths form a sequence  $$l_1\ge l_2\ge...\ge l_s>0$$(the empty sequence for $Y=\emptyset$), we define the $\frak{sl}(\infty)$-module $V_Y$ as a direct limit $\varinjlim_{n\ge s} V_Y(n)$: here $V_Y(n)$ denotes a simple finite-dimensional $\frak{sl}(n)$-module with highest weight $$l_1\ge l_2\ge...\ge l_s>0\ge0\ge...\ge0$$ where the sequence has length $n$~(for $Y=\emptyset$ the highest weight of $V_Y(n)$ equals 0). The $\frak{sl}(n)$-module $V_Y(n)$ is isomorphic to a simple direct summand of the tensor product $$\operatorname{S}^{l_1}(V(n))\otimes\operatorname{S}^{l_2}(V(n))\otimes....\otimes\operatorname{S}^{l_s}(V(n)),$$ and the direct limit $\varinjlim_{n\ge s} V_Y(n)$ is clearly well defined up to isomorphism. Similarly, for $\frak g(\infty)=\frak{sl}(\infty)$, we define $(V_Y)_*$ as a direct limit $\varinjlim_{n\ge s}(V_Y(n))^*$.

Finally, we set $$I(x, y, s, t, l_1,..., l_s, r_1,..., r_t)=I(x, y, Y_l, Y_r):=\Ann_{\Ue(\frak{sl}(\infty))}~(V_{Y_l}\otimes(\Sa^\cdot(V))^{\otimes x}\otimes(\Lambda^\cdot(V))^{\otimes y}\otimes (V_{Y_r})_*)$$ where $x, y\in\mathbb Z_{\ge0}$, $Y_l$ and $Y_r$ are Young diagrams with respective column lengths $l_1,..., l_s$ and $r_1,..., r_t$.

The classification of primitive ideals of $\Ue(\frak{sl}(\infty))$ can now be stated as follows.
\begin{theorem}\label{T1} All ideals $I(x, y, s, t,  l_1,..., l_s, r_1,..., r_t)$ are primitive and nonzero, and any nonzero primitive ideal $I$ of $\Ue(\frak{sl}(\infty))$ equals exactly one of these ideals.\end{theorem}

Since Proposition~4.8  in~\cite{PP2} asserts that the primitive ideals $I(x, y, Y_l, Y_r)$ are precisely the integrable primitive ideals of $\Ue(\frak{sl}(\infty))$, in order to prove Theorem~\ref{T1} it is enough to prove the following.
\begin{theorem}\label{T1b} Every nonzero primitive ideal of $\Ue(\frak{sl}(\infty))$ is integrable.\end{theorem}

\section{Corollaries, examples and further results}
A brief review of basic facts concerning  splitting Cartan and Borel subalgebras of $\frak{sl}(\infty)$ (including the definitions), as well as the roots of $\frak{sl}(\infty)$, see in~\cite{PP3}. For any splitting Borel subalgebra $\frak b\subset\frak{sl}(\infty)$ with fixed Cartan subalgebra $\frak h$, there is a well-defined notion of simple $\frak b$-highest weight module $L_\frak b(\lambda)$ with highest weight $\lambda\in\frak h^*$. Given a weight $\lambda\in\frak h^*$, by definition $L_\frak b(\lambda)$ is the unique simple quotient of the induced module $\Ue(\frak{sl}(\infty))\otimes_{\Ue(\frak  b)}\mathbb F_\lambda$, where $\mathbb F_\lambda$ is a one-dimensional $\frak b$-module one which $\frak h$ acts through the weight $\lambda$.

There is a class of Borel subalgebras of $\frak{sl}(\infty)$, which we call ideal. 
A quick definition of an {\it ideal Borel subalgebra} $\frak b$ of $\frak{sl}(\infty)$ is as follows: this is a splitting Borel subalgebra of $\frak{sl}(\infty)$ for which any simple object in the category of tensor modules $\mathbb T_{\frak{sl}(\infty)}$ defined in~\cite{DPS} is a $\frak b$-highest weight $\frak{sl}(\infty)$-module. In~\cite[Section~2.2]{PP2} (see also~\cite[Section~5]{PP3}) we have given a description of ideal Borel subalgebras in terms of their roots~\footnote{It is worth pointing out that ideal Borel subalgebras are not the ``most obvious'' Borel subalgebras in a matrix representation of $\frak{sl}(\infty)$. In particular, if one thinks of $\frak{sl}(\infty)$ as the finitary matrices extending infinitely down and to the right, then the Borel subalgebra of upper triangular matrices is not ideal.}.

The following is an analogue of Duflo's Theorem.
\begin{theorem}\label{T1c}Let $\frak b$ be an ideal Borel subalgebra of $\frak{sl}(\infty)$ with a fixed splitting Cartan subalgebra $\frak h\subset\frak b$.\\
a)For any primitive ideal $I$ of $\Ue(\frak{sl}(\infty))$ there exists $\lambda\in\frak h^*$ such that $I=\Ann_{\Ue(\frak{sl}(\infty))}L_\frak b(\lambda)$.\\
b) If $I=I(0, 0, s, t, l_1,..., l_s, r_1,..., r_t)$ then the weight $\lambda\in\frak h^*$, such that $I=\Ann_{\Ue(\frak{sl}(\infty))}L_\frak b(\lambda)$, is unique.
\end{theorem}
\begin{proof}Part a) is implied by Theorem~\ref{T1} and~\cite[Theorem 3.1]{PP2}.

Part b) follows directly from a more general uniqueness result of A.~Sava~\cite{S}, see also~\cite{PP3}.\end{proof}
\begin{Ex0} Recall that the roots of $\frak{sl}(\infty)$ have the form $\{\varepsilon_i-\varepsilon_j\}_{i\ne j\in\mathbb Z_{>0}}$, where $\varepsilon_i$ are certain standard vectors, see~\cite[Appendix~A]{PP3}. Consider the Borel subalgebra $\frak b\subset\frak{sl}(\infty)$ with positive roots \begin{center}$\{\varepsilon_i-\varepsilon_j\mid (i>j)\&(2\nmid i, j),$ or $(i<j)\&(2\mid i, j),$ or $(2\nmid i)\&(2\mid j)$\}.\end{center}
This Borel subalgebra is ideal~\cite{PP3}. Given an arbitrary primitive ideal $I(x, y, s, t, l_1,..., l_s, r_1,..., r_t)$, the weight $\lambda$ in Theorem~\ref{T1c}a) can be chosen as

$$\sum\limits_{1\le i\le x}i\alpha\varepsilon_{2i-1}+\sum_{1\le j\le s}l_i\varepsilon_{2i+2x-1}+y(\sum_{k\ge1}\varepsilon_{2k-1})+\sum_{1\le j\le t}r_{t+1-j}\varepsilon_{2j}$$
for an arbitrary $\alpha\in\mathbb F\backslash\mathbb Q$. Moreover, one  can show that, if $x=0$, then the above weight $\lambda$ is unique with the property $I=\Ann_{\Ue(\frak{sl}(\infty))}L_\frak b(\lambda)$.

However, this uniqueness result is specific to the chosen ideal Borel subalgebra $\frak b$. Indeed, consider the Borel subalgebra $\frak b'$ of $\frak{sl}(\infty)$ with positive roots
\begin{center}$\{\varepsilon_i-\varepsilon_j\mid (i>j)\&(2\nmid i, j)\&(i\ne 1)$ or $(i>j)\&(2\mid i, j)$ or $(2\nmid i)\&(2\mid j)$\}.\end{center}
Then $\frak b'$ is also ideal, but one can check that $\Ann_{\Ue(\frak{sl}(\infty))}L_{\frak b'}(\lambda)=\Ann_{\Ue(\frak{sl}(\infty))}L_{\frak b'}(\lambda')$ where
$$\lambda=\sum\limits_{i\ge 1}\varepsilon_i,\hspace{10pt}\lambda'=\sum\limits_{i\ge 2}\varepsilon_i.$$

Finally, if $x=y=0$ and $\frak b''$ is any ideal Borel subalgebra then the weight $\lambda$, such that
$$I(0, 0, s, t, l_1,..., l_s, r_1,..., r_t)=\Ann_{\Ue(\frak{sl}(\infty))} L_{\frak b''}(\lambda),$$ is unique. Moreover, in this case the $\frak{sl}(\infty)$-module $L_{\frak b''}(\lambda)$ is isomorphic to the socle of the tensor product $V_{Y_l}\otimes (V_{Y_r})_*$ (see in~\cite{PSt} the proof that this tensor product has simple socle).\hspace{235pt}$\diamondsuit$\end{Ex0}
\section{Proof of Theorem~\ref{T1b}}
The proof consists of many reduction steps which we will go through one-by-one.

By an ideal of an associative algebra we always mean a two-sided ideal. We set $$\Ue:=\Ue(\frak{sl}(\infty)),\hspace{10pt} \Ue_n:=\Ue(\frak{sl}(n))\subset\Ue.$$ For an ideal $I\in\Ue$, we put $I_n:=I\cap\Ue_n$ for $n\ge2$.

Let $A$ be an associative algebra and $M$ be an $A$-module. We say that $M$ is {\it integrable} if, for any finitely generated subalgebra $A'\subset A$ and any $m\in M$, we have $$\dim (A'\cdot m)<\infty.$$We define an ideal $I\subset A$ to be {\it integrable} if $I=\Ann_{A}M$ for an integrable $A$-module $M$. 
An ideal $I\subset A$ is {\it locally integrable} if, for any finitely generated subalgebra $A'\subset A$, the ideal $I\cap A'$ is an integrable ideal of $A'$.

It is easy to see that an ideal $I\subset\Ue$ is locally integrable iff, for every $n\ge2$, the ideal $I_n$ is an intersection of ideals of finite codimension in $\Ue_n$.

Theorem~\ref{T1b} is a direct corollary of the following two statements:
\begin{proposition}\label{Pprtolint}If $I\subset\Ue$ is a primitive ideal then $I$ is locally integrable.\end{proposition}
\begin{proposition}\label{Plinttoint} If $I$ is a locally integrable ideal then $I$ is integrable. \end{proposition}
A stronger version of Proposition~\ref{Pprtolint} is proved in Subsection~\ref{SSprt1}, and in Subsections~\ref{SSprt2}-\ref{SScomb} we prove a chain of statements which imply Proposition~\ref{Plinttoint}.
\subsection{Proof of Proposition~\ref{Pprtolint}}\label{SSprt1}
Let $I$ be an ideal of an associative algebra $A$. We denote by $\sqrt I$ the intersection of all primitive ideals of $A$ containing $I$. One may note that $\sqrt I$ is the pullback in $A$ of the Jacobson radical of the ring $A/I$. If $I$ is a primitive ideal then $I=\sqrt I$.

It is clear that Proposition~\ref{Pprtolint} is a consequence of the following statement.
\begin{proposition}\label{Pjtolint} Let $I$ be an ideal of $\Ue$. Then $\sqrt I$ is a locally integrable ideal.\end{proposition}
To prove this proposition, we first need the following alternative description of $\sqrt I$.
\begin{lemma}\label{Lalg} Let $I\subset A$ be an ideal and let the cardinality of $\mathbb F$ exceed the $\mathbb F$-dimension of $A$. Then the following conditions on an element $z\in A$ are equivalent:

1) $z\in \sqrt I$,

2) for every $a\in A$ there is $k\in\mathbb Z_{>0}$, such that $(az)^k\in I$.\end{lemma}
\begin{proof} The fact that 1) implies 2) follows from~\cite[p.~344,~Corollary~1.8]{MCR}. We now show that 2) implies 1).

Let $z\in A$ satisfy 2), and let $\bar z$ be the image of $z$ in $A/I$. Assume to the contrary that there exists a simple $A/I$-module $M$ such that $\bar z\cdot M\ne0$. Pick $m\in M$  with $\bar z\cdot m\ne0$. There exists $\bar a\in A$ so that $\bar a\cdot(\bar z\cdot m)=m$. Let $k\in\mathbb Z_{>0}$ satisfy $(\bar a\bar z)^k=0$. Then
$$0=(\bar z(\bar a\bar z)^k)\cdot m=\bar z\cdot m\ne0.$$
This contradicts our assumption that $\bar z\cdot M\ne0$. Hence $\bar z\in\sqrt I$.
\end{proof}
Next, we prove
\begin{lemma}\label{Lrad} Let $I$ be an ideal of $\Ue$. Then there exists $r\in\mathbb Z_{>0}$ such that, for any $n\ge r$ and any primitive ideal $J(n)\subset \Ue_n$ containing $I_n$, the intersection $J(n)\cap\Ue_{n-r}$ is an integrable ideal in $\Ue_{n-r}$.\end{lemma}
\begin{proof} For any $n\ge2$ and any $r'\in\mathbb Z_{\ge0}$ we put $$\frak{sl}(n)^{\le r'}:=\{x\in\frak{sl}(n)\mid \exists\lambda\in\mathbb F : \operatorname{rk}(x-\lambda{\bf 1}_n)\le r'\},$$
where $\operatorname{rk}$ refers to the rank of a matrix, and ${\bf 1}_n$ is the identity $n\times n$-matrix, cf.~\cite{PP1}. For any ideal $J\subset\Ue_n$, we denote by $\Var(J)$ the algebraic variety in $\frak{sl}(n)^*$ corresponding to the graded ideal $$\operatorname{gr} J\subset\Sa^\cdot(\frak{sl}(n)).$$

According to~\cite[Theorem~3.3]{PP1}, there exists $r\in\mathbb Z_{\ge0}$ such that $\Var(I_n)=\frak{sl}(n)^{\le r}$; here we identify $\frak{sl}(n)$ and $\frak{sl}(n)^*$ via the Killing form. Since $J(n)\supset I_n$, we have $$\Var(J(n))\subset\Var(I_n)\subset\frak{sl}(n)^{\le r}.$$

A well known theorem of A.~Joseph claims that the associated variety of a primitive ideal equals the closure of a nilpotent coadjoint orbit, see~\cite{Jo}. In our case, the  nilpotent coadjoint orbits are identified with the conjugacy classes of nilpotent $n\times n$-matrices. These conjugacy classes are in 1-1 correspondence with the partitions of $n$: the partition attached to a conjugacy class comes from the Jordan normal form of a representative of this class. In this way we attach to $J(n)$ a partition of $n$. By $p(n)$ we denote the partition conjugate to that partition, and let $r(n)$ be the difference between $n$ and the maximal element of $p(n)$. It is crucial that the inclusion $\Var(J(n))\subset\frak{sl}(n)^{\le r}$ implies $r(n)\le r$, see~\cite[Subsection 4.3]{PP2}.

Next, we need an explicit classification of primitive ideals of $\Ue_n$, see~\cite{PP2}. Namely, a primitive ideal $J$ of $\Ue_n$ is determined by its intersection with the center of $\Ue_n$, together with the left cell in the integral Weyl group attached to this intersection, see for example~\cite[Section 6]{LO}. The intersection of $J$ with the center of $\Ue_n$ can be encoded by an unordered $n$-tuple $a_1',..., a_n'\in \mathbb F$. The integral Weyl group is isomorphic to a direct product of symmetric groups, and the factors of this direct product are parametrized by the equivalence classes of elements of $\{1,..., n\}$ with respect to the equivalence relation
$$i\sim j\Leftrightarrow a_i'-a_j'\in\mathbb Z.$$
The left cells of the integral Weyl group of $J$ are in 1-1 correspondence with the collections of Young tableaux so that the entries of the $i$th tableau are the elements of the $i$th equivalence class in $\{1,..., n\}$, see~\cite[p.~172]{BV1}.

Inserting $a_i'$ instead of $i$ in all these semistandard tableux, we attach to any primitive ideal $J$ the datum
$$\cup_{t\in\mathbb F/\mathbb Z}\{a^t_{1, 1}, a^t_{2, 1}, ..., a^t_{l_1^t, 1}; a^t_{1, 2}, a^t_{2, 2}, ..., a^t_{l_2^t, 2};...; a^t_{1, h_t},..., a^t_{l_{h_t}^t, h_t}\}$$
where

1) $h_t\ne0$ only for a finite subset of $\mathbb F/\mathbb Z$ and $\Sigma_{t, j} l_j^t=n$,

2) $a^t_{i, j}\in\mathbb F$ and the image of $a^t_{i, j}$ in $\mathbb F/\mathbb Z$ equals $t$,

3) $a^t_{i, j}-a^t_{i', j}\in\mathbb Z_{>0}$ for all $t\in\mathbb F/\mathbb Z, 1\le j\le h_t, 1\le i<i'\le l_j^t$,

4) $l_j^t\le l_{j'}^t$ for all $t\in\mathbb F/\mathbb Z, 1\le j<j'\le h_j$.

5) $a^t_{i, j}-a^t_{i, j'}\in\mathbb Z_{\ge0}$ for all $t\in\mathbb F/\mathbb Z, 1\le j\le h_t, 1\le i\le l_j^t, 1\le j'\le j$.\\
Here $h_t$ is the height of the $t$th tableau, $l_1^t,..., l^t_{h_t}$ are the row lengths of the $t$th tableau, and $a_{i,j}^t$ is the $i$th entry of the $j$th row of the $t$th tableau.

We now assume that the above datum corresponds to $J(n)$, and let $t_1,..., t_s$ be the elements of $\mathbb F/\mathbb Z$ for which $h_t\ne0$. Then the parts of $p(n)$ are all nonzero elements in the sequence$$l_1^{t_1}, l_2^{t_1},..., l_{h_{t_1}}^{t_1}, l_1^{t_2}, l_2^{t_2},..., l_{h_{t_2}}^{t_2},..., l_1^{t_s},..., l_{h_{t_s}}^{t_s}$$(repetitions are possible). Therefore, $$r(n)=n-\max\limits_{t\in\mathbb F/\mathbb Z,~1\le j\le h_t}l_j^t.$$

Denote by $\lambda^*(n)$ the $\frak{sl}(n)$-weight corresponding to the sequence
\begin{center}$a^{t_1}_{1, 1}, a^{t_1}_{2, 1}, ..., a^{t_1}_{l_1^{t_1}, 1}, \hspace{10pt}a^{t_1}_{1, 2}, a^{t_1}_{2, 2}, ..., a^{t_1}_{l_2^{t_1}, 2},\hspace{10pt}...,$\\
$a^{t_2}_{1, 1}, a^{t_2}_{2, 1}, ..., a^{t_2}_{l_1^{t_2}, 1}, \hspace{10pt}a^{t_2}_{1, 2}, a^{t_2}_{2, 2}, ..., a^{t_2}_{l_2^{t_2}, 2},\hspace{10pt}...,$\\{~}...{~}\\
$a^{t_s}_{1, 1}, a^{t_s}_{2, 1}, ..., a^{t_s}_{l_1^{t_s}, 1},\hspace{10pt} a^{t_s}_{1, 2}, a^{t_s}_{2, 2}, ..., a^{t_s}_{l_2^{t_s}, 2},\hspace{10pt}...,$\end{center} by $\rho$ the weight corresponding to the sequence
$$0, 1, 2,..., n-1,$$ and put $\lambda(n):=\lambda^*(n)+\rho$.

Then $J(n)=\Ann_{\Ue_n}L(\lambda(n))$ where $L(\lambda(n))$ is a simple $\frak{sl}(n)$-module with highest weight $\lambda(n)$. This can be seen for instance by following the algorithm in~\cite[Subsection 4.2]{PP2}. It is clear that $L(\lambda(n))|_{\frak{sl}(l_j^{t_i})}$ is an integrable module for each root subalgebra $\frak{sl}({l_j^{t_i}})$~of $\frak{sl}(n)$ corresponding to the subsequence $a^{t_i}_{1, j},..., a^{t_i}_{l_j^{t_i}, j}$. This implies that $J(n)\cap \Ue(\frak{sl}({l_j^t}))$ is an integrable ideal.

Let $t(n), j(n)$ be such that $$l_{j(n)}^{t(n)}=\max\limits_{t\in\mathbb F/\mathbb Z, 1\le j\le h_t}l_j^t.$$
Then, by the above, $J(n)\cap \Ue(\frak{sl}(l_{j(n)}^{t(n)}))$ is an integrable ideal of $\Ue(\frak{sl}(l_{j(n)}^{t(n)}))$. On the other hand, note that $r(n)=n- l_{j(n)}^{t(n)}$ and the Lie subalgebra $\frak{sl}(l_{j(n)}^{t(n)})$ is conjugate in $\frak{sl}(n)$ to the fixed Lie subalgebra $\frak{sl}(n-r(n))$ for which $$\Ue_{n-r(n)}=\Ue(\frak{sl}(n-r(n))).$$ Therefore, the ideals $J(n)\cap\Ue_{n-r(n)}\subset\Ue_{n-r(n)}$ and $J(n)\cap\Ue(\frak{sl}(l_{j(n)}^{t(n)}))\subset\Ue(\frak{sl}(l_{j(n)}^{t(n)}))$ are identified under conjugation. Consequently, $J(n)\cap\Ue_{n-r(n)}$ is an integrable ideal of $\Ue_{n-r(n)}$. Since $n-r\le n-r(n)$, it follows that $J(n)\cap\Ue_{n-r}$ is an integrable ideal of $\Ue_{n-r}$.\end{proof}
\begin{proof}[Proof of Proposition~\ref{Pjtolint}] Lemma~\ref{Lrad} implies the existence of $r\ge0$ so that
$\sqrt{I_{n+r}}\cap\Ue_{n}$ is an integrable ideal of $\Ue_{n}$ for any $n\ge2$.
 Next, Lemma~\ref{Lalg} shows that $(\sqrt I)_{n}=\cap_{n'\ge n}\sqrt{I_{n'}}$. However, $$\cap_{n'\ge n}\sqrt{I_{n'}}=(\cap_{n'\ge n+r}\sqrt{I_{n'}})\cap\Ue_n.$$
Being integrable in $\Ue_{n'-r}$, the ideal $\sqrt{I_{n'}}\cap\Ue_{n'-r}$ is an intersection of ideals of finite codimension in $\Ue_{n'-r}$, hence $(\sqrt I)_n=(\cap_{n'\ge n+r}\sqrt{I_{n'}})\cap\Ue_n$ is an intersection of ideals of finite codimension in $\Ue_n$. This means that the ideal $(\sqrt I)_n$ is integrable for $n\ge2$.
\end{proof}

\subsection{Locally integrable ideals and p.l.s..}\label{SSprt2} Let $I$ be a locally integrable  ideal of $\Ue$. As we pointed out, for every $n\ge2$, $I_n\subset\Ue_n$ is an intersection of ideals of finite codimension in $\Ue_n$. Therefore, $I_n$ is the intersection of annihilators of finite-dimensional $\frak{sl}(n)$-modules. Since any finite-dimensional $\frak{sl}(n)$-module is semisimple, it follows that $I_n$ is an intersection of annihilators of simple finite-dimensional $\Ue_n$-modules.

Let $\Irr_n$ denote the set of isomorphism classes of simple finite-dimensional $\frak{sl}(n)$-modules. We put $$Q(I)_n:=\{[M]\in\Irr_n\mid (I\cap\Ue_n)\subset\Ann_{\Ue_n}M\}$$
where $M$ stands for a simple finite-dimensional $\frak{sl}(n)$-module and $[M]$ denotes the isomorphism class of $M$. For any $n'\ge n$ and any subset $Q_{n'}\subset\Irr_{n'}$ we denote by $(Q_{n'})|_{\frak{sl}(n)}$ the set of isomorphism classes of all simple $\frak{sl}(n)$-submodules of the $\frak{sl}(n')$-modules $M$ with $[M]\in Q_{n'}$. It is clear that $Q(I)_{n'}|_{\frak{sl}(n)}\subset Q(I)_n$. This leads to the following definition.

Let $Q=\{Q_2, Q_3,..., Q_n,...\}$ be a collection of subsets $Q_2\subset \Irr_2, Q_3\subset\Irr_3,..., Q_n\subset\Irr_n,...$. We call $Q$ a {\it precoherent local system} ({\it p.l.s.} for short) if $Q_{n'}|_{\frak{sl}(n)}\subset Q_n$ for all $n'\ge n\ge 2$. By definition, $Q$ is a {\it coherent local system} ({\it c.l.s.} for short) if $Q_{n'}|_{\frak{sl}(n)}=Q_n$ for all $n'\ge n\ge 2$. The notion of c.l.s. has been introduced by A.~Zhilinskii~\cite{Zh1}.

The collection $\{Q_n(I)\}_{n\ge2}$ defined above is immediately seen to be a p.l.s.. We denote this p.l.s. by $Q(I)$. Conversely,
given a p.l.s. $Q$, we assign to $Q$ the ideal
$$I(Q):=\cup_{n'\ge n}(\cap_{[M]\in Q_{n'}}\Ann_{\Ue_{n'}}M)\subset\Ue.$$
Clearly, $I(Q)$ is a locally integrable ideal $I$ of $\Ue$. Moreover, $I(Q(I))=I$ for any locally integrable ideal $I\subset\Ue$. This reduces Proposition~\ref{Plinttoint} to the following statement.
\begin{proposition}\label{Pdef}If $Q$ is a p.l.s. then $I(Q)$ is an integrable ideal.\end{proposition}
\begin{remark}One can show that $Q(I)$ is a c.l.s. whenever $I$ is an integrable ideal. Therefore, Proposition~\ref{Pdef} implies that $Q(I)$ is in fact a c.l.s. under the weaker assumption that $I$ is a locally integrable ideal of $\Ue$.\end{remark}

We say that two p.l.s. $Q, Q'$ are {\it equivalent} if there exists $n\ge 2$ such that $Q_{n'}=Q'_{n'}$ for all $n'\ge n$. It is clear that if $Q$ and $Q'$ are equivalent, then $I(Q)=I(Q')$. It is known that if $Q$ is a c.l.s., then $I(Q)$ is an integrable ideal~\cite{PP1}. Thus, in order to prove Proposition~\ref{Pdef}, it suffices to prove the following.
\begin{proposition}\label{Ppcls2}For any p.l.s. $Q$ there exists a c.l.s. $Q'$ such that $Q$ and $Q'$ are equivalent.\end{proposition}
Next, we reduce Proposition~\ref{Ppcls2} to a purely combinatorial statement. We call a nonincreasing sequence $\lambda_1\ge\lambda_2\ge...\ge\lambda_n$ of integers a {\it $\mathbb Z$-partition} of {\it width} $\sharp\lambda:=n$ ($\mathbb Z$-partitions of width $n$ are precisely the integral dominant weights of $\frak{gl}(n)$). We then identify $\Irr_n$ with the set of $\mathbb Z$-partitions of width $n$ modulo the equivalence relation $$(\lambda_1\ge ...\ge \lambda_n)\sim(\lambda_1+D\ge...\ge\lambda_n+D), \hspace{10pt}D\in\mathbb Z.$$
By $V_\lambda$ we denote a simple finite-dimensional $\frak{sl}(n)$-module corresponding to the $\mathbb Z$-partition $\lambda$ with $\sharp\lambda=n$. By a slight abuse of notation we write $\lambda\in\Irr_{\sharp\lambda}$.

The classical Gelfand-Tsetlin rule claims that, for $\mathbb Z$-partitions $\lambda$ and $\mu$ with $\sharp\lambda=n$, $\sharp\mu=n-1$, the following conditions are equivalent:

$\bullet$ $\Hom_{\frak{sl}(n-1)}(V_\mu, V_\lambda|_{\frak{sl}(n-1)})\ne 0$,

$\bullet$ there exists $D\in\mathbb Z$ such that $\lambda_1\ge\mu_1+D\ge\lambda_2\ge\mu_2+D\ge...\ge\lambda_{n-1}\ge\mu_{n-1}+D\ge\lambda_n$.\\
We write $\lambda>\mu$ whenever these conditions hold.

For general $\mathbb Z$-partitions $\lambda$ and $\mu$ with $\sharp\lambda\ge\sharp\mu$, the Gelfand-Tsetlin rule implies that the following conditions are equivalent:

$\bullet$ $\Hom_{\frak{sl}({\sharp \mu})}(V_\mu, V_\lambda|_{\frak{sl}(\sharp\mu)})\ne 0$,

$\bullet$ there exists a sequence of $\mathbb Z$-partitions $\lambda=\lambda^0, \lambda^1,..., \lambda^{n-m}=\mu$ such that
$$\lambda=\lambda^0>\lambda^1>...>\lambda^{n-m}=\mu$$
and $\sharp \lambda^i=\sharp\lambda-i$. We write $\lambda\succ\mu$ whenever these latter conditions hold, and say that $\lambda$ {\it dominates} $\mu$.

We can now restate the definitions of c.l.s. and p.l.s. as follows.
Let $Q=\{Q_2, Q_3, ..., Q_n\}$ be a collection of subsets $Q_2\subset\Irr_2, Q_3\subset\Irr_3,..., Q_n\subset\Irr_n,...$. Then\\
a) the following conditions are equivalent:

$\bullet$ $Q$ is a p.l.s.,

$\bullet$ for all $\lambda, \mu$ such that $\lambda\succ\mu$ and $\lambda\in Q_{\sharp\lambda}$, we have $\mu\in Q_{\sharp\mu}$;\\
b) the following conditions are equivalent:

$\bullet$ $Q$ is a c.l.s.,

$\bullet$ $Q$ is a p.l.s. and for every $\mu\in Q_{\sharp\mu}$ there is $\lambda\in Q_{\sharp\mu+1}$ such that $\lambda\succ\mu$.\\
We denote by $Q^\vee(\lambda)$ the largest p.l.s. $Q$ which does not contain (the equivalence class of) a given $\mathbb Z$-partition $\lambda$. It is clear that

$Q^\vee(\lambda)_n$ consists of all $\mathbb Z$-partitions of width $n$ for $n<\sharp\lambda$,

$Q^\vee(\lambda)_{\sharp\lambda}$ consists of all $\mathbb Z$-partitions of width $\sharp\lambda$ except $\lambda$,

$Q^\vee(\lambda)_n$ consists of all $\mathbb Z$-partitions $\mu$ of width $n$ such that $\mu\not\succ\lambda$ for all $n>\sharp\lambda$.\\
We are now ready to state
\begin{proposition}\label{Pmain}For any $\mathbb Z$-partition $\lambda$, the p.l.s. $Q^\vee(\lambda)$ is equivalent to the c.l.s.
$$Q(\lambda):=\cup_{1\le k< l\le \sharp\lambda} Q(k, l, \lambda_k-\lambda_l)$$
where $Q(k, l, \lambda_k-\lambda_l)$ is the c.l.s. defined by the formula
$$Q(k, l, \lambda_k-\lambda_l)_m:=\{\mu\in\Irr_m\mid \mu_k-\mu_{m-\sharp\lambda+l}<\lambda_k-\lambda_l\}.$$
\end{proposition}
The next subsection is devoted to the proof of Proposition~\ref{Pmain}. We conclude this subsection by showing how Proposition~\ref{Pmain} implies Proposition~\ref{Ppcls2}, and therefore ultimately Proposition~\ref{Pjtolint}.
\begin{proof}[Proof of Proposition~\ref{Ppcls2}]
Let $Q$ be a p.l.s. Then, clearly

$$Q=\cap_{\lambda\not\in Q} Q^\vee(\lambda).$$
According to Proposition~\ref{Pmain}, a p.l.s. of the form $Q^\vee(\lambda)$ is equivalent to the c.l.s. $Q(\lambda)$. The lattice of c.l.s. is artinian~\cite{Zh1}, and therefore we conclude that $Q$ is equivalent to
a c.l.s. $$Q(\lambda_1)\cap....\cap Q(\lambda_s)$$
for some finite set of elements $\lambda_1,..., \lambda_s\not\in Q$.\end{proof}

\subsection{Combinatorics of $\mathbb Z$-partitions}\label{SScomb}
It is clear that Proposition~\ref{Pmain} is implied by the following.
\begin{proposition}\label{Lgts2}Let $\lambda$ and $\mu$ be $\mathbb Z$-partitions such that $\sharp\mu\ge 4\sharp\lambda$. Then  the following conditions are equivalent:

1) $\mu\succ\lambda$,

2) $\mu_k-\mu_{\sharp\mu-\sharp\lambda+l}\ge\lambda_k-\lambda_l$ for any $1\le k<l\le\sharp\lambda$.
\end{proposition}
In the proof of Proposition~\ref{Lgts2} we need the following three lemmas.
\begin{lemma}\label{Lgts}Let $\lambda$ and $\mu$ be $\mathbb Z$-partitions such that $\sharp\mu\ge\sharp\lambda$, $\lambda_1=\mu_1, \lambda_{\sharp\lambda}=\mu_{\sharp\mu}$. Then \begin{center} $\mu_i\ge \lambda_i\ge \mu_{\sharp\mu-\sharp\lambda+i}$ for $1\le i\le \sharp\lambda$\end{center}
implies $\mu\succ\lambda$.\end{lemma}
\begin{proof}If $\sharp\lambda=\sharp\mu$ then clearly $\lambda=\mu$. For $\sharp\mu>\sharp\lambda$, arguing by induction, it clearly suffices to show the existence of a $\mathbb Z$-partition $\mu'$ such that

$\bullet$ $\mu>\mu', \sharp\mu'=\sharp\mu-1$,

$\bullet$ $\mu'_1=\mu_1, \mu'_{\sharp\mu'}=\mu_{\sharp\mu}$,

$\bullet$ $\mu'_i\ge\lambda_i\ge\mu'_{\sharp\mu'-\sharp\lambda+i}$ for $1\le i\le \sharp\lambda$,\\
This is straightforward and we leave the details to the reader.\end{proof}
\begin{lemma}\label{Lgts25}Let $\lambda$ and $\mu$ be $\mathbb Z$-partitions such that $\sharp\mu\ge 2\sharp\lambda$. Then the conditions

a) for every $k, l\in\mathbb Z_{\ge1}$ such that~$1\le k<l\le\sharp\lambda$ we have $\mu_k-\mu_{\sharp\mu-\sharp\lambda+l}\ge\lambda_k-\lambda_l$,

b) there are $k, l\in\mathbb Z_{\ge1}$ such that $1\le k<l\le\sharp\lambda$ and $\mu_k-\mu_{\sharp\mu-\sharp\lambda+l}=\lambda_k-\lambda_l$\\
imply $\mu\succ\lambda$.
\end{lemma}
\begin{proof}Condition a) and b) implies via Lemma~\ref{Lgts} the existence of $\mathbb Z$-partitions $\mu^0, \mu^1,..., \mu^{m-n}$ such that

$\bullet$ $\mu^{i}\succ\mu^{i+1}$, $\sharp\mu^i=l-k+1+(\sharp\mu-\sharp\lambda-i)$,

$\bullet$ $\mu^0=(\mu_k\ge...\ge\mu_{\sharp\mu-\sharp\lambda+l})$,

$\bullet$ $\mu^{\sharp\mu-\sharp\lambda}=(\lambda_k\ge ...\ge \lambda_l)$,

$\bullet$ $\lambda_k=\mu^0_1=\mu_1^1=\mu^2_1=...=\mu^{\sharp\mu-\sharp\lambda}_0,$ and $\lambda_l=\mu^0_{\sharp\mu^0}=\mu^1_{\sharp\mu^1}=...=\mu^{\sharp\mu-\sharp\lambda}_{l-k+1}$.\\
We set

$$\hat{\mu^i}_j:=\begin{cases}\mu_j\mbox{~for~}j< k-i,\\
\lambda_j+(\mu_k-\lambda_k)\mbox{~for~}k-i\le j\le k,\\
\mu^i_{j-k}\mbox{~for~}k<j\le (\sharp\mu-i)-\sharp\lambda+l,\\
\lambda_{j-(\sharp\mu-\sharp\lambda)}+(\mu_{\sharp\mu-\sharp\lambda+l}-\lambda_l)\mbox{~for~}(\sharp\mu-i)-\sharp\lambda+l< j\le \sharp\mu-\sharp\lambda+l,\\
\mu_j\mbox{~for~}j>\sharp\mu-\sharp\lambda+l.\end{cases}$$
One can easily check that

$\bullet$ $\hat \mu_{i+1}\succ\hat \mu_i$,

$\bullet$ $\hat \mu^0=\mu$ (here it is crucial that $\sharp\mu-\sharp\lambda\ge \sharp\lambda\ge\max(k, \sharp\lambda-l+1))$,

$\bullet$ $\mu^{\sharp\mu-\sharp\lambda}=\lambda$.\\
Thus $\mu\succ\lambda$.\end{proof}
\begin{lemma}\label{Lfir} Let $\lambda, \mu$ be $\mathbb Z$-partitions, and $i\in\mathbb Z_{>0}$ be a positive integer such that

$$\sharp \lambda \le i\le \sharp\mu-i,\hspace{10pt}\mu_i-\mu_{\sharp\mu-i+1}\ge\lambda_1-\lambda_{\sharp\lambda}.$$
Then $\mu\succ\lambda$.\end{lemma}
\begin{proof} Put $\mu':=(\mu_1,..., \mu_i, \mu_{\sharp\mu-i+1},..., \mu_{\sharp\mu})$. It is clear that $\mu\succ\mu'$, and that $\mu'$ and $\lambda$ satisfy the conditions of Lemma~\ref{Lfir} as well. Therefore without loss of generality we can assume that $\mu=\mu'$, and thus that $i=\sharp\mu-i$. We can also assume that $\lambda_1=\mu_i$. This implies that $\lambda_{\sharp\lambda}\ge\mu_{\sharp\mu-i+1}=\mu_{i+1}.$

Next, we observe that the following sequence of $\mathbb Z$-partitions each element dominates the next:
$$\mu_1\ge...\ge\mu_i\ge\mu_{i+1}\ge...\ge\mu_{\sharp\mu}$$
$$\mu_1\ge...\ge\mu_{i-1}\ge\lambda_{\sharp\lambda}\ge\mu_{i+1}\ge...\ge\mu_{\sharp\mu-1}$$
$$...$$
$$\mu_1\ge...\ge\mu_{i-k}\ge\lambda_{\sharp\lambda-k+1}\ge...\ge\lambda_{\sharp\lambda}\ge\mu_{i+1}\ge...\ge\mu_{\sharp\mu-k}$$
$$...$$
$$\mu_1\ge...\ge\mu_{i-\sharp\lambda}\ge\lambda_1\ge...\ge\lambda_{\sharp\lambda}\ge\mu_{i+1}\ge...\ge\mu_{\sharp\mu-\sharp\lambda}.$$
The last $\mathbb Z$-partition dominates $\lambda$, hence $\mu\succ\lambda$.
\end{proof}
\begin{proof}[Proof of Proposition~\ref{Lgts2}]It is clear that 1) implies 2). We show now that 2) implies 1). To do this, we assume to the contrary that 2) holds and $\mu\not\succ\lambda$. We claim that this contradicts Lemma~\ref{Lgts25}.

Indeed, consider the $\mathbb Z$-partitions $\mu^{r, s}$, for $r, s\le\sharp\lambda$, where
$$\mu_i^{r, s}=\mu_{i+r},\mbox{~for~}1\le i\le \sharp\mu-s-r,~\mu^{0, 0}=\mu.$$
It is clear that $\mu\succ\mu^{r, s}$. Lemma~\ref{Lfir} implies that $\mu_{\sharp\lambda}-\mu_{\sharp\mu-\sharp\lambda+1}<\lambda_1-\lambda_{\sharp\lambda}$, and in particular that $\mu^{\sharp\lambda-1, \sharp\lambda-1}$ does not satisfy condition 2) of Proposition~\ref{Lgts2} considered as an abstract condition on a partition $\mu'$ instead of $\mu$. Therefore, since $\mu=\mu^{0, 0}$ satisfies this condition, there exist $r, s<\sharp\lambda$ so that $\mu^{r, s}$ satisfies this condition, and $\mu^{r+1, s}$ or $\mu^{r, s+1}$ does not satisfy this condition. These two cases are very similar, and we consider only the first one (leaving the second one to the reader).

We put $\mu':=\mu^{r, s}$ and assume that $\mu^{r+1, s}\not\succ\lambda$. Then there exist $k, l$, for $1\le k< l\le\sharp\lambda$, such that

$$\mu'_k-\mu'_{\sharp\mu-\sharp\lambda+l}\ge\lambda_k-\lambda_l,\hspace{10pt}\mu'_{k+1}-\mu'_{\sharp\mu-\sharp\lambda+l}< \lambda_k-\lambda_l.$$
Without loss of generality we can assume that $l$ is chosen so that the value of
$$\mu_{\sharp\mu-\sharp\lambda+l}+\lambda_k-\lambda_l$$
is maximal. This implies that, for
$$\mu'':=\mu'_1, \mu'_2,..., \mu'_{k-1}, \mu'_k+\lambda_l-\lambda_k, \mu'_{k+1}, \mu'_{k+2},..., \mu'_{\sharp\mu-r-s-1},$$
we have $\mu\succ\mu''$, and all conditions of Lemma~\ref{Lgts25} are satisfied for the pair $(\lambda, \mu'')$ (here it is crucial that $\mu+2-(r+s)\ge 2\lambda$). Hence $\mu''\succ\lambda$, and we have the desired contradiction.\end{proof}
\section{Appendix: The inclusion order on primitive ideals.}
As explained in~\cite[Subsection~7.3]{PP1}, a c.l.s. for $\frak{sl}(\infty)$ can be encoded by a pair of nonincreasing sequences
\begin{equation}\label{Eseqlr}p_1\ge p_2\ge p_3\ge...{\rm~and~}q_1\ge q_2\ge q_3\ge...\end{equation}
of elements of $\mathbb Z_{\ge0}\sqcup\{+\infty\}$ with common limit
$$\lim\limits_{i\to\infty}p_i=\lim\limits_{i\to\infty}q_i=m,$$
see~\cite{Zh1}, see also~\cite[Proposition~7.5 and Subsection~7.3]{PP1}. We denote by ${\rm cls}(p_1, p_2,... ; q_1, q_2,...)$ the c.l.s. attached to the pair of sequences~(\ref{Eseqlr}).
The inclusion order on c.l.s. is described by the following theorem due to A.~Zhilinskii.
\begin{theorem}[{\cite[Subsection 2.5]{Zh1},~see~also~\cite[Subsection~7.3]{PP1}}]\label{Tiap}Let $\{p_i, q_i\}_{i\ge 1}$ and $\{p_i', q_i'\}_{i\ge 1}$ be pairs of nonincreasing sequences with respective limits $m, m'$ as in~(\ref{Eseqlr}). Then the following conditions are equivalent:

a) ${\rm cls}(p_1', p_2',...; q_1', q_2',...)\subset{\rm cls}(p_1, p_2,...; q_1, q_2,...)$,

b) there exists $a, b\in\mathbb Z_{\ge0}$ such that $a+b=m-m'$ and
$$p_i'\le p_i-a,\hspace{10pt}q_i'\le q_i-b.$$
\end{theorem}
Fix $x, y, s, t, l_1, ..., l_s, r_1,..., r_t$, and consider $I:=I(x, y, l_1,..., l_s, r_1,..., r_t)$ as in Section~\ref{Smr}. Put
\begin{equation}p^c_i:=\begin{cases}+\infty&{\rm~if~} 1\le i\le c\\y+l_{i-c}&{\rm~if~} c+1\le i\le c+s\\y&{\rm~if~} i>c+s+1\\\end{cases},\hspace{10pt}q^d_i:=\begin{cases}+\infty&{\rm~if~} 1\le i\le d\\y+r_{i-d}&{\rm~if~} d+1\le i\le d+t\\y&{\rm~if~} i>d+t+1\\\end{cases}\label{Eseqpq}\end{equation}
for any $c, d\in\mathbb Z_{\ge0}$.
\begin{proposition}\label{Pordz}We have

a) $I=I(x, y, l_1,..., l_s, r_1,..., r_t)=I({\rm cls}(p_1^x, p_2^x,...; q_1^0, q_2^0,...)),$

b) $I=I(x, y, l_1,..., l_s, r_1,..., r_t)=I({\rm cls}(p_1^c, p_2^c,...; q_1^d, q_2^d,...))$ for all $c, d$ such that $c+d=x$,

c) $Q(I)=\cup_{c+d=x} {\rm cls}(p_1^c, p_2^c,...; q_1^d, q_2^d,...)$.
\end{proposition}
\begin{proof}Part a) follows from the definition of $I(x, y, l_1,..., l_s, r_1,..., r_t)$, see also~\cite[Theorem~7.9]{PP1}. Part b) follows from the discussion in~\cite[Subsection~7.4]{PP1}, see formula
$$I(v, w, Q_f)=I(v+w, 0, Q_f)$$
in the notation of~\cite{PP1}.

Part c) is implied by~\cite[Lemma 7.6c)]{PP1}.\end{proof}
Alltogether, this allows us to provide an explicit inclusion criterion for a pair of primitive ideals.
\begin{theorem}\label{Tord} The following conditions are equivalent:

a) there is a (not necessarily strict) inclusion
$$I(x, y, s, t, l_1,..., l_s, r_1,..., r_t)\subset I(x', y', s', t',  l_1',..., l'_{s'}, r'_1,..., r'_{t'}),$$

b) $x\ge x', y\ge y'$ and, for some $a, b, c, d\in\mathbb Z_{\ge0}$ with $c+d=x-x', a+b=y-y'$, the inequalities
$$ l_{i}-a\ge l'_{i+c},\hspace{10pt}r_{j}-b\ge r'_{j+d},$$
where
$$l_i=0{\rm~if~}i>s,\hspace{10pt}l'_i=0{\rm~if~}i>s',\hspace{10pt}r_j=0{\rm~if~}j>t,\hspace{10pt}r'_j=0{\rm~if~}j>t'$$
are satisfied for all $i, j\ge 0$.
\end{theorem}
\begin{proof} Let $\{p_i^c, q_i^d\}_{i\ge1}$, $\{(p')_i^c, (q')_i^c\}$ be defined by~(\ref{Eseqpq}).

We would like to show first that b) implies a). It is straightforward to check that
$${\rm cls}((p')_1^{x'}, (p')_2^{x'},...; (q')_1^{0}, (q')_2^{0},...)\subset {\rm cls}(p_1^{x'+c}, p_2^{x'+c},...; q_1^{d}, q_2^{d},...)$$
by Theorem~\ref{Tiap}. This together with Proposition~\ref{Pordz} a-b) implies a).

Now we show that a) implies b). It is clear that
$$Q(I(x', y', s', t',  l_1',..., l'_{s'}, r'_1,..., r'_{t'}))\subset Q(I(x, y, s, t, l_1,..., l_s, r_1,..., r_t)).$$
According to Proposition~\ref{Pordz} c), this implies
$$\cup_{c+d=x'} {\rm cls}((p')_1^c, (p')_2^c,...; (q')_1^d, (q')_2^d,...)\subset \cup_{c+d=x} {\rm cls}(p_1^c, p_2^c,...; q_1^d, q_2^d,...).$$
In particular, we have
$${\rm cls}((p')_1^{x'}, (p')_2^{x'},...; (q')_1^0, (q')_2^0,...)\subset \cup_{c+d=x} {\rm cls}(p_1^c, p_2^c,...; q_1^d, q_2^d,...).$$
The c.l.s. ${\rm cls}((p')_1^{x'}, (p')_2^{x'},...; (q')_1^0, (q')_2^0,...)$ is irreducible~\cite[Definition~I.I.I]{Zh1}, and therefore
$${\rm cls}((p')_1^{x'}, (p')_2^{x'},...; (q')_1^0, (q')_2^0,...)\subset {\rm cls}(p_1^c, p_2^c,...; q_1^d, q_2^d,...)$$
for some $c, d$ such that $c+d=x$, see~\cite[Proposition~I.I.2]{Zh1}. Next, by Theorem~\ref{Tiap} we have $c\ge x'$, and moreover there exist $a, b\in\mathbb Z_{\ge0}$ such that $a+b=y-y'$ and the condition b) of Theorem~\ref{Tiap} holds. Thus b) holds for $a, b~(a+b=y-y')$ and $c-(x'), d~((c-x')+d=x-x')$.

\end{proof}
\begin{corollary}The lattice of two-sided ideals of $\Ue(\frak{sl}(\infty))$ satisfies the ascending chain condition.\end{corollary}
\begin{corollary}The augmentation ideal $I(0, 0, 0, 0)$ is the only maximal ideal of $\Ue(\frak{sl}(\infty))$.\end{corollary}
\begin{proof} The statement  is implied by Theorems~\ref{T1b} and ~\ref{Tord}.\end{proof}
\section{Acknowledgements}
 We thank A.~Baranov for inroducing us to the work of A.~Zhilinskii several years ago. We also thank D.~Rumynin and S.~Sierra for discussions which led finally to the statement and proof of Lemma~\ref{Lalg}. In addition, we are grateful for A.~Fadeev for pointing out some some details in the proofs which required improvement. Both authors acknowledge partial support through DFG Grant~PE 980/6-1. The second author was also supported by Leverhulme Trust Grant RPG-2013-293, and thanks Jacobs University for its hospitality.

\end{document}